\documentclass[11pt,a4paper,reqno]{amsart}
\usepackage{stmaryrd}
\usepackage[english]{babel}
\usepackage{pst-grad} 
\usepackage{pst-plot} 
\usepackage{pstricks}
\usepackage{amsmath,amssymb,mathrsfs,amsthm, mathtools}
\usepackage{tikz} 
\usepackage{mathcomp,wasysym}  
\usepackage{graphicx}  
\usepackage[all]{xy} \xyoption{arc} \xyoption{color}
\usepackage{epsfig}
\usepackage{cite}
\usepackage[a4paper,
left=2.5cm, right=2.5cm,
top=3cm, bottom=3cm]{geometry}

\newcommand{\TT}{\mathbb T}
\newcommand{\pat}{\partial_t}
\newcommand{\pax}{\partial_x}

\normalsize
\normalsize
\usepackage[bookmarks,colorlinks]{hyperref}

\newtheorem{theorem}{Theorem}

\title{A nonlocal model describing tumor angiogenesis}

\author[R. Granero-Belinch\'{o}n]{Rafael Granero-Belinch\'{o}n$^*$}
\email{rafael.granero@unican.es}
\address{Departamento  de  Matem\'aticas,  Estad\'istica  y  Computaci\'on,  Universidad  de Cantabria.  Avda.  Los  Castros  s/n,  Santander,  Spain.}

\begin{document}
\begin{abstract}
In this paper we study the onset of angiogenesis and derive a new model to describe it. This new model takes the form of a nonlocal Burgers equation with both diffusive and dispersive terms. For a particular value of the parameters, the equation reduces to
$$
\pat p-\frac{1}{2}(-\Delta)^{(\alpha-1)/2}H \pat p=-\frac{1}{2}(-\Delta)^{\alpha/2}  p+ p\pax p-\partial_x p,
$$
where $H$ denotes the Hilber transform. In addition to the derivation of the new model, we also prove a number of well-posedness results. Finally, some preliminary numerics are shown. These numerics suggest that the dynamics of the equation is rich enough to have solutions that blow up in finite time.
\end{abstract}

\subjclass[2010]{}


\maketitle
{\small
\tableofcontents}

\allowdisplaybreaks
\section{Introduction}
The motion of cells in response to a different values of chemical concentrations is known as chemotaxis. When the chemical is diffusible, the resulting problem has been heavily studied by many different authors since the pioneer work of Patlak \cite{patlak1953random}. In the case of non-diffusible signals that are deposited by the cells the resulting system of partial differential equations is
\begin{equation}\label{eq:1}
\left\{\begin{aligned}
\pat u&=-(-\Delta)^{\alpha/2} u+\chi\pax \left(u\frac{\pax w}{w}\right),\\
\pat w&=uw,
\end{aligned}\right.\text{ for }x\in\TT,\,t\geq0,
\end{equation}
This system was proposed by Othmers \& Stevens \cite[Equation (78)]{stevens1997aggregation} to model cells moving randomly that deposit a non-diffusible signal that modifies the local environment for subsequent movement. For instance, one can consider the movement of myxobacteria or ants. Indeed, myxobacteria produce slime over which other myxobacteria can move easily and ants can follow
trails left by other ants. Such a chemotactic motion is a crucial step in many different biological phenomena ranging from slime mold aggregation \cite{keller1970initiation} to the formation of new blood vessels from pre-existing blood vessels in a process that is called angiogenesis \cite{levine2000mathematical}. 

Angiogenesis is a very complicated phenomenon that appears in many different biological situations. Due to its importance, it has been studied by many different authors in the mathematical community (see for instance \cite{corrias2003chemotaxis,friedman2002stability,granero2017global,granero2017fractional,levine2000mathematical,levine2001mathematical,
li2011hyperbolic,li2010nonlinear,li2011asymptotic,wang2008shock} and the references therein). Angiogenesis is also a key step during tumor growth. Roughly speaking (see \cite{levine2001mathematical} for a more detailed description), endothelial cells are located in the inner part of blood vessels, lying over a part of the extracellular matrix called the basal lamina. Then, during certain stage of tumor growth, the tumor induce angiogenesis by releasing angiogenic factors. Activated by these chemicals, endothelial cells in nearby capillaries thicken and accumulate in certain regions. Following activation, cell-released proteases degrade the basal lamina adjacent to the activated endothelial cells. The endothelial cells loosen their contact with their neighbor cells and begin to penetrate the basal lamina. Then the vessel wall dilates as the endothelial cells accumulate and a sprout is formed. This sprout is composed of endothelial cells where the angiogenic stimulus has reached a threshold. This new capillary network then supplies nutrients to the tumor colony and allow for further for tumor expansion. 

The purpose of this paper is to derive and study new mathematical models to describe angiogenesis. In that regards, system \eqref{eq:1} serves as a starting point. In particular, \eqref{eq:1} were also derived by Levine, Sleeman \& Nilsen-Hamilton \cite[Equation (7.2.1)]{levine2001mathematical} (to obtain \eqref{eq:1} from equation (7.2.1) take $\theta\equiv 0$ and rename the parameters and unknowns) to describe the initial step of capillary formation in tumor angiogenesis (see also Levine, Sleeman \cite{sleeman1997system}). Similar equations were also derived in \cite[Equations (4.1) and (4.2)]{levine2000mathematical} and \cite[Equation (2.2.8)]{levine2001mathematical}). In these works, the movement of endothelial cells is modeled using the idea of reinforced random walks and the extracellular matrix is modeled with only one of its components, fibronectin \cite{levine2001mathematical}. Fibronectin plays an important role in the attachment and migration of cells. In this framework, $u$ describes the concentration of endothelial cells and $w$ describes the density of capillary wall, represented by fibronectin \cite{levine2000mathematical}. The core of the idea is that the accumulation of endothelial cells in certain region along a capillary is stimulated by low levels of fibronectin \cite{levine2001mathematical}. 

In this paper we derive the following Burgers equation with a dispersive term
\begin{align}\label{eq:Burgers}
\pat p-\frac{1}{2}(-\Delta)^{(\alpha-1)/2}H \pat p=-\frac{1}{2\varepsilon}(-\Delta)^{\alpha/2}  p+\chi p\pax p-\frac{\beta}{\varepsilon}\partial_x p,
\end{align}
where $H$ is the Hilbert transform and $(-\Delta)^{s/2}$ is the fractional Laplacian. These two are singular integral operators that can also be defined using Fourier variables (see below for proper definitions). Setting
$$
\beta=\frac{\chi-1}{2},
$$
$\varepsilon$ a small parameter and 
$$
\partial_x \log(w)=\varepsilon p
$$
equation \eqref{eq:Burgers} appears as an asymptotic model of \eqref{eq:1} for near homogeneous values of endothelial cell density
$$
u(x,t)=1+\varepsilon h(x,t).
$$

Burgers equations with nonlocal terms of diffusive type such as
$$
\pat p=-(-\Delta)^{\alpha/2}  p+ p\pax p,
$$
have been the topic of study of different research groups in the last years. In terms of the dychotomomy global well-posedness vs finite time blow up phenomena, Kiselev, Nazarov \& Shterenberg  \cite{kiselev2008blow} and Dong, Du \& Li \cite{dong2009finite} established the global existence for large values of $\alpha$ together with a finite time singularity result for small values of $\alpha$ (see also \cite{burczak2016critical,chickering2021asymptotically}). Other properties of the solution have also been the goal of different research projects \cite{alibaud2010asymptotic,biler1998fractal,karch2008convergence}.

In the case of dispersive regularizations of Burgers equations, Linares, Pilod \& Saut \cite{linares2014dispersive} and Molinet, Pilod \& Vento \cite{molinet2018well} studied the global solvability of a Whitham type equations
$$
\pat p=(-\Delta)^{\alpha/2}\pax p+ p\pax p.
$$
Dispersive Burgers equations are known to have singularities in finite time \cite{castro2010singularity,hur2017wave,saut2020wave}. Particular mention must be done to the dispersionless Burgers-Hilbert equation
$$
\pat p=H p+ p\pax p.
$$
There, the singularities occur\cite{castro2010singularity,saut2020wave} but they do at later times than suggested by standard energy estimates \cite{hunter2012enhanced,hunter2015long}. Also, stability of travelling waves \cite{castro2021stability} and global existence of weak solutions are known \cite{bressan2014global}.

\subsection{Notation}
We introduce the Hilbert transform
$$
H f( \alpha) = \frac{1}{2\pi} P.V. \int_\mathbb{T} \frac{f (y)}{\tan((x - y)/2) } d y  \,.
$$
This singular integral operator is the following multiplier operator in the Fourier variables
$$
\hat f(k) = {\dfrac{1}{\sqrt{2\pi}}} \int_{\mathbb{T}} f(x) \ 
e^{-ikx}dx,
$$
namely
$$
\widehat{Hf}(k)=-i\text{sgn}(k) \hat{f}(k).
$$
Finally, we introduce the fractional Laplacian operator,  
$$
\widehat{(-\Delta)^{\alpha/2} f}(k)=|k|^\alpha\hat{f}(k).
$$
The functional spaces that we will use in this paper are the $L^2$-based homogeneous Sobolev spaces
\begin{equation*}\label{Sobhomo}
H^\alpha(\mathbb{T})=\left\{u\in L^2(\mathbb{T}),\quad \|u\|_{H^\alpha(\mathbb{T})}^2:=\sum_{k\in\mathbb{Z}}|k|^{2\alpha}|\widehat{u}(k)|^2<\infty
\right\}.
\end{equation*}
and the homogeneous Wiener spaces $A^\alpha(\mathbb{T})$ as
\begin{equation}\label{Wienerhomo}
A^\alpha(\mathbb{T})=\left\{u\in L^1(\mathbb{T}),\quad \|u\|_{A^\alpha(\mathbb{T})}:=\sum_{k\in\mathbb{Z}} |k|^\alpha|\widehat{u}(k)|<\infty\right\}.
\end{equation}

\section{Derivation}
Let us begin with the derivation of \eqref{eq:Burgers} from \eqref{eq:1}. We start with the system \eqref{eq:1} written for $$
q=\partial_x \log(w),
$$
\begin{equation}\label{eq:1v1}
\left\{\begin{aligned}
\pat u&=-(-\Delta)^{\alpha/2} u+\chi\pax (uq),\\
\pat q&=\pax u,
\end{aligned}\right.\text{ for }x\in\TT,\,t\geq0.
\end{equation}
We fix $\varepsilon$ a small parameter. After changing variables as follows
$$
u=1+\varepsilon h,\quad q=\varepsilon p
$$
we find that
\begin{align*}
\pat h&=-(-\Delta)^{\alpha/2} h+\varepsilon\chi\pax (hp)+\chi\pax p,\\
\pat p&=\pax h,
\end{align*}
where $0\leq \alpha\leq 2$. We use far field variables 
$$
\xi=x-t,\quad \tau=\varepsilon t,
$$
so
$$
\partial_t=\varepsilon \partial_\tau-\partial_\xi,\quad \partial_x=\partial_\xi.
$$
Then, the previous system reads
\begin{align*}
\varepsilon \partial_\tau h -\partial_\xi h&=-(-\Delta)^{\alpha/2} h+\varepsilon \chi \partial_\xi (hp)+\chi\partial_\xi p,\\
\varepsilon \partial_\tau p -\partial_\xi p&=\partial_\xi h,
\end{align*}
Differentiating the equation for $p$ in the $\tau$ variable, we find that
$$
\varepsilon^2 \partial_\tau^2 p -\varepsilon\partial_\xi \partial_\tau p=\varepsilon\partial_\xi \partial_\tau h.
$$
Due to the equation for $h$, we find that
$$
\varepsilon^2 \partial_\tau^2 p -\varepsilon\partial_\xi \partial_\tau p=\partial_\xi^2 h-(-\Delta)^{\alpha/2} \partial_\xi h+\varepsilon \chi \partial_\xi^2 (hp)+\chi\partial_\xi^2 p.
$$
Using that
$$
h=-p+\varepsilon\int\partial_\tau p d\xi,
$$
we find that
$$
\varepsilon^2 \partial_\tau^2 p -\varepsilon\partial_\xi \partial_\tau p=\varepsilon \partial_\xi\partial_\tau p -\partial_\xi^2 p-(-\Delta)^{\alpha/2} (\varepsilon \partial_\tau p -\partial_\xi p)+\varepsilon\chi\partial_\xi^2 \left(\left(-p+\varepsilon\int\partial_\tau p d\xi\right)p\right)+\chi\partial_\xi^2 p.
$$
Then, if we neglect terms of order $O(\varepsilon^2)$, we obtain the equation
\begin{align*}
-2\varepsilon\partial_\tau\partial_\xi p=(-\Delta)^{\alpha/2}\partial_\xi p-\varepsilon(-\Delta)^{\alpha/2} \partial_\tau p -\varepsilon \chi\partial_\xi^2 (p^2)+(\chi-1)\partial_\xi^2 p.
\end{align*}
Integrating in $\xi$ and changing back to our previous notation for the independent variables, we conclude 
\begin{align}\label{eq:Burgers2}
\pat p-\frac{1}{2}(-\Delta)^{(\alpha-1)/2}H \pat p=-\frac{1}{2\varepsilon}(-\Delta)^{\alpha/2}  p+\chi p\pax p-\frac{\chi-1}{2\varepsilon}\partial_x p,
\end{align}
which is \eqref{eq:Burgers} after renaming the parameters. Once we have derived this model, the rest of the paper is devoted to its mathematical study. Thus, from this point onwards, and for the sake of generality, we consider that the parameter $\varepsilon$ can take arbitrary values. To simplify the notation we consider the new variable
$$
p=\chi p
$$
and consider the equation
\begin{align}\label{eq:Burgers2v2}
\pat p-\frac{1}{2}(-\Delta)^{(\alpha-1)/2}H \pat p=-\frac{1}{2\varepsilon}(-\Delta)^{\alpha/2}  p+ p\pax p-\frac{\beta}{\varepsilon}\partial_x p.
\end{align}
From \eqref{eq:Burgers2v2}, we can further compute
\begin{multline*}
\left(1+\frac{1}{2}(-\Delta)^{(\alpha-1)/2}H\right)\left(1-\frac{1}{2}(-\Delta)^{(\alpha-1)/2}H\right) \pat p\\
=-\left(1+\frac{1}{2}(-\Delta)^{(\alpha-1)/2}H\right)\frac{1}{2\varepsilon}(-\Delta)^{\alpha/2}  p+\left(1+\frac{1}{2}(-\Delta)^{(\alpha-1)/2}H\right)\partial_x\left(\frac{p^2}{2}\right)\\
-\frac{\beta}{\varepsilon}\left(1+\frac{1}{2}(-\Delta)^{(\alpha-1)/2}H\right)\partial_x p.
\end{multline*}
Using
$$
\left(1+\frac{1}{2}(-\Delta)^{(\alpha-1)/2}H\right)\left(1-\frac{1}{2}(-\Delta)^{(\alpha-1)/2}H\right)=1+\frac{1}{4}(-\Delta)^{\alpha-1}
$$
so $p$ solves
\begin{multline}\label{eq:Burgers3}
\pat p+\frac{1}{4}(-\Delta)^{\alpha-1}\pat p\\
=-\frac{\beta+1}{2\varepsilon}(-\Delta)^{\alpha/2}  p-\frac{1}{4\varepsilon}(-\Delta)^{\alpha-1/2}H  p+\partial_x\left(\frac{p^2}{2}\right)+(-\Delta)^{\alpha/2}\left(\frac{p^2}{4}\right)-\frac{\beta}{\varepsilon}\partial_x p.
\end{multline}
We observe that this equation resembles the classical BBM equation \cite{benjamin1972model} or the Buckley-Leverett equation \cite{buckley1942mechanism,burczak2016generalized} (see also \cite{montgomery2001finite}).

\section{The case $\alpha=0$}
In the case $\alpha=0$, equation \eqref{eq:Burgers3} reads as follows
$$
\pat p+\frac{1}{4}(-\Delta)^{-1}\pat p
=-\frac{1}{2\varepsilon}p-\frac{1}{4\varepsilon}(-\Delta)^{-1/2}H  p+\partial_x\left(\frac{p^2}{2}\right)+\frac{p^2}{4}
-\frac{\beta}{\varepsilon}\left(\partial_x p+\frac{1}{2} p\right).
$$
Taking $-\Delta$ of the previous equation and using that
$$
(-\Delta)^{1/2}H=-\partial_x,
$$
we compute
\begin{equation}\label{eq:alpha0}
-\Delta\pat p+\frac{1}{4}\pat p
=\frac{1+\beta}{2\varepsilon}\Delta p+\frac{1}{4\varepsilon}\pax p-\partial^3_x\left(\frac{p^2}{2}\right)-\Delta\left(\frac{p^2}{4}\right)
+\frac{\beta}{\varepsilon}\partial_x^3 p.
\end{equation}
For this equation we have the following well-posedness theorem:
\begin{theorem}[Strong well-posedness for $\alpha=0$]
Let $p_0\in H^2$ be a zero-mean initial data, $\beta>-1$ and $\varepsilon>0$ be fixed constants. Then there exists a unique local solution to \eqref{eq:alpha0}
$$
p\in C([0,T_{max}),H^{2})\cap L^2(0,T_{max};H^{2})
$$
for a small enough $0<T_{max}\ll1$. Furthermore, there exists $0<c_0$ such that if
$$
\|\Delta p_0\|_{L^2}+\frac{1}{4}\|\pax p_0\|_{L^2}^2\leq c_0,
$$
then we have that there exists a unique global solution to \eqref{eq:alpha0}
$$
p\in C([0,T),H^2)\cap L^2(0,T;H^{2})\quad\forall\,T>0
$$
emanating from this initial data. Furthermore, the solution verifies
$$
\|\Delta p\|_{L^2}+\frac{1}{4}\|\pax p\|_{L^2}^2+\frac{\beta+1}{2\varepsilon}\int_0^t\|\Delta p(s)\|_{L^2}^2ds\leq C(p_0).
$$
\end{theorem}
\begin{proof}
The proof follows from appropriate energy estimates after a standard regularization using for instance a Galerkin approximation (see \cite{AGM,GO} for a similar approach using mollifiers). Thus, we focus on obtaining the \emph{bona fide} energy estimates. We start noticing that the zero-mean property is propagated in time. Testing \eqref{eq:alpha0} against $-\Delta p$, we find
$$
\frac{1}{2}\frac{d}{dt}\left(\|\Delta p\|_{L^2}^2+\frac{1}{4}\|\pax p\|_{L^2}^2\right)
=-\frac{1+\beta}{2\varepsilon}\|\Delta p\|_{L^2}^2+\int_\mathbb{T}\Delta p\partial^3_x\left(\frac{p^2}{2}\right)dx+\int_\mathbb{T}\Delta p\Delta\left(\frac{p^2}{4}\right)dx.
$$
We have that
\begin{align}
I_1&=\int_\mathbb{T}\Delta p\partial^3_x\left(\frac{p^2}{2}\right)dx\nonumber\\
&=-\frac{1}{2}\int_\mathbb{T}\pax^3 p\partial^2_x\left(p^2\right)dx\nonumber\\
&=-\frac{1}{2}\int_\mathbb{T}\pax^3 p\left(2p \pax^2 p+2(\pax p)^2\right)dx\nonumber\\
&=\frac{5}{2}\int_\mathbb{T}\pax p (\pax^2 p)^2dx\label{eq:I1}
\end{align}
Similarly,
\begin{align}
I_2&=\int_\mathbb{T}\Delta p\Delta\left(\frac{p^2}{4}\right)dx\nonumber\\
&=\frac{1}{4}\int_\mathbb{T} \pax^2 p\left(2p \pax^2 p+2(\pax p)^2\right)dx\nonumber\\
&=\frac{1}{2}\int_\mathbb{T} p (\pax^2p)^2 pdx\label{eq:I2}.
\end{align}
Then, we define
$$
E(t)=\|\Delta p(t)\|_{L^2}^2+\frac{1}{4}\|\pax p(t)\|_{L^2}^2.
$$
The zero-mean property leads us to
$$
\|p\|_{L^\infty}\leq 2\pi\|\partial_x p\|_{L^\infty}.
$$ 
The previous ineqaulity, H\"older's inequality and Sobolev embedding, allow us to conclude the inequality
$$
\frac{d}{dt}E(t)\leq -\frac{1+\beta}{\varepsilon}\|\Delta p\|_{L^2}^2+C\|\pax p\|_{L^\infty}\|\Delta p\|_{L^2}^2\leq CE(t)^{3/2},
$$
where we have used
$$
\|\pax p\|_{L^\infty}^2\leq C\|\pax p\|_{L^2}\|\Delta p\|_{L^2}\leq CE(t).
$$
The local existence follows from the previous inequality using a classical regularization procedure (see, for instance, \cite{AGM,burczak2016generalized,GO}). The uniqueness follows from a standard contradiction argument together with the regularity of the solution. Similarly, using the previous computations, we can find the inequality
$$
\frac{d}{dt}E(t)\leq \left(C\sqrt{E(t)}-\frac{1+\beta}{\varepsilon}\right)\|\Delta p\|_{L^2}^2.
$$
As a consequence, if 
$$
E(0)\ll1
$$
then 
$$
\frac{d}{dt}E(t)\leq0
$$
and the solution is global.
\end{proof}
We can simplify the previous equation \eqref{eq:alpha0} and find that
$$
\pat p
=\frac{\beta+1}{2\varepsilon}\mathcal{K}\pax^2 p+\frac{1}{4\varepsilon}\mathcal{K}\pax p-\mathcal{K}\partial^3_x\left(\frac{p^2}{2}\right)-\mathcal{K}\Delta\left(\frac{p^2}{4}\right)
+\frac{\beta}{\varepsilon}\mathcal{K}\partial_x^3 p.
$$
where
$$
\widehat{\mathcal{K}}(k)=\frac{1}{\frac{1}{4}+k^2}.
$$
Written in this form, the equation is ready to be implemented using a Fourier collocation method to discretize in space. Then, the integration in time can be carried out using a standard Runge-Kutta procedure. In particular, after simulating the case $\alpha=0$ using a variable step Runge-Kutta 4-5 with $N=2^{12}$ spatial nodes, $\varepsilon=1$, $\beta=2$ and initial data
$$
p(x,0)=-2\sin(4x),
$$
we obtain the solution plotted in figures \ref{figalpha0} and \ref{figalpha01}. There we can see numerical evidence of finite time singularity formation as the solution seems to steepen up and the derivative seems to blow up.
\begin{figure}[h]
\begin{center}
\includegraphics[scale=0.5]{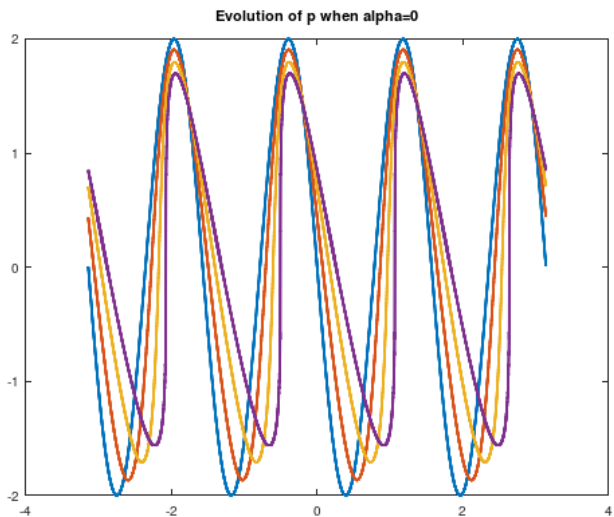}
\end{center}
\vspace{-.1 in}
\caption{The solution for different times.}
\label{figalpha0}
\end{figure}
\begin{figure}[h]
\begin{center}
\includegraphics[scale=0.4]{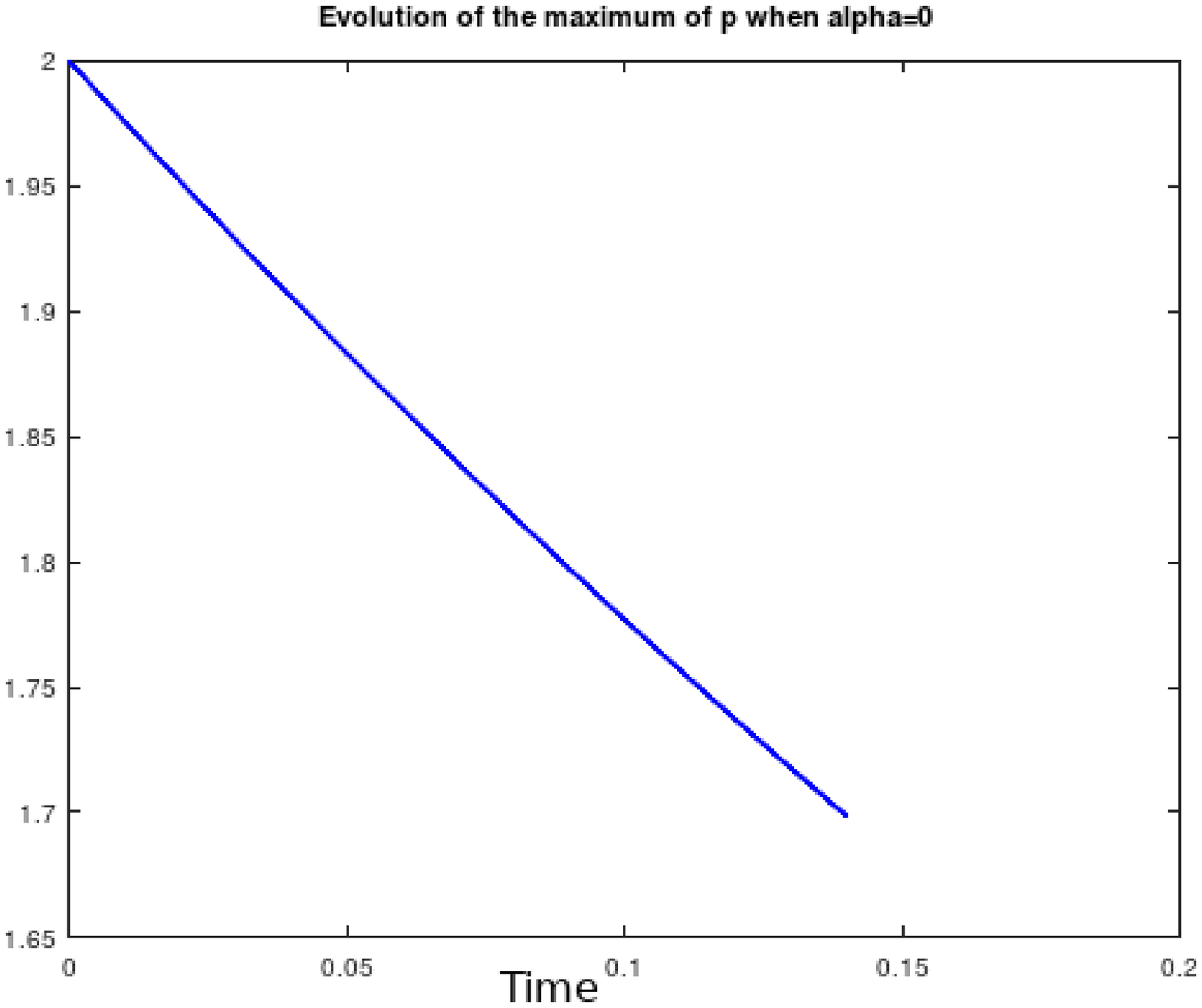} \includegraphics[scale=0.4]{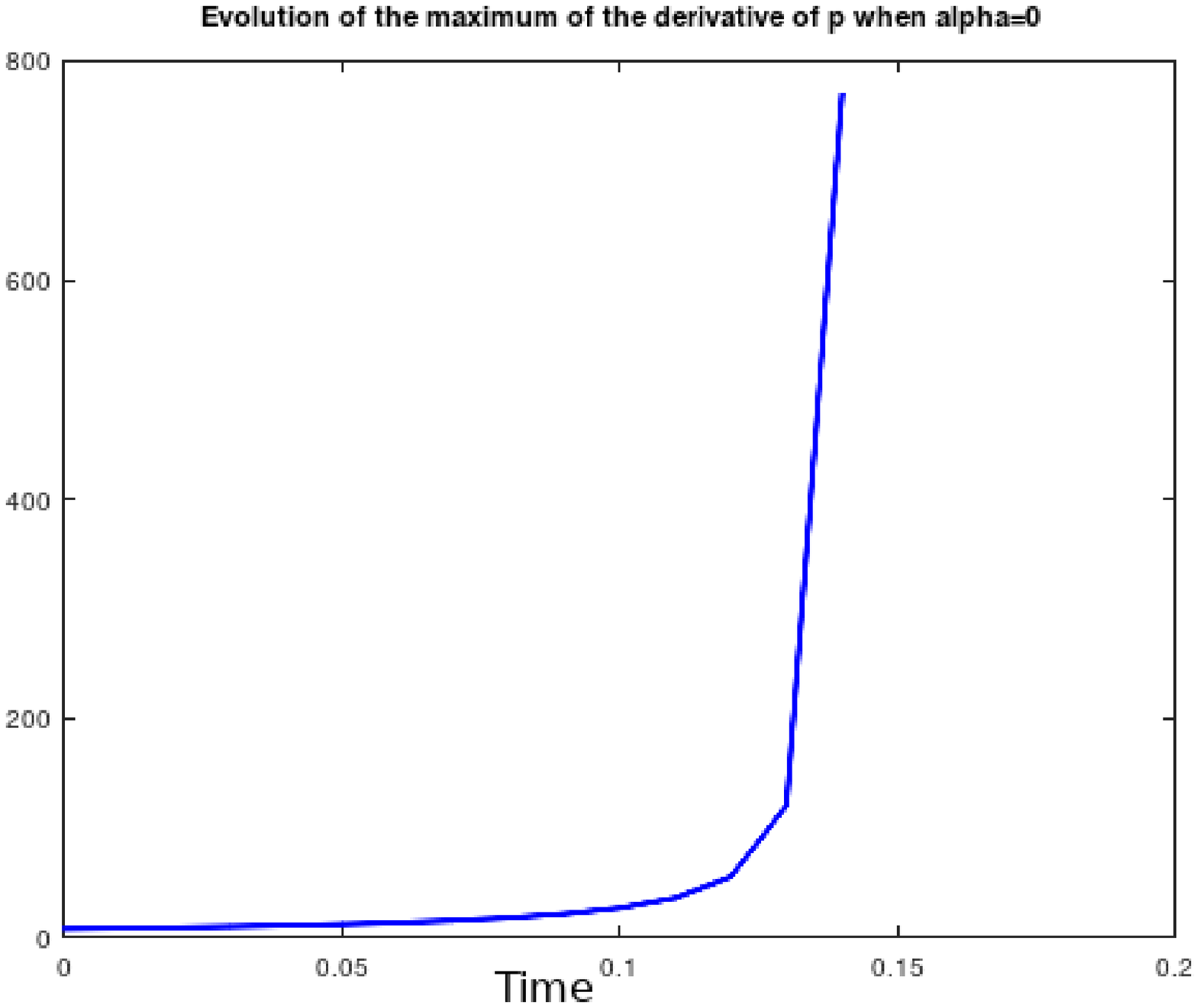}
\end{center}
\vspace{-.1 in}
\caption{a) $\|p(t)\|_{L^\infty}$ as a function of time. b) $\|\pax p(t)\|_{L^\infty}$ as a function of time.}
\label{figalpha01}
\end{figure}

\section{The case $\alpha=1$}
In this section we consider the case $\alpha=1$. This case is critical in the sense that every differential operator, regardless of its parabolic or hyperbolic character, is of order one. Then \eqref{eq:Burgers3} reduces to
$$
\frac{5}{4}\pat p=-\frac{1}{2\varepsilon}(-\Delta)^{1/2}  p-\frac{1}{4\varepsilon}(-\Delta)^{1/2}H  p+\partial_x\left(\frac{p^2}{2}\right)+(-\Delta)^{1/2}\left(\frac{p^2}{4}\right)
-\frac{\beta}{\varepsilon}\left(\partial_x p+\frac{1}{2}(-\Delta)^{1/2} p\right).
$$
Recalling
$$
(-\Delta)^{1/2}H=-\partial_x,
$$
we find the equation
\begin{equation}\label{eq:alpha1}
\frac{5}{4}\pat p=-\frac{\beta+1}{2\varepsilon}(-\Delta)^{1/2}  p+\frac{\left(\frac{1}{4}-\beta\right)}{\varepsilon}\pax p+p\pax p+(-\Delta)^{1/2}\left(\frac{p^2}{4}\right).
\end{equation}
\begin{theorem}[Strong well-posedness for $\alpha=1$]
Let $p_0\in H^2$ be a zero-mean initial data, $\beta>-1$ and $\varepsilon>0$ be fixed constants. Then there exists $0<c_0$ such that if
$$
\|p_0\|_{A^0}\leq c_0,
$$
then we have that there exists a unique global solution to \eqref{eq:alpha1}
$$
p\in C([0,T),H^2)\cap L^2(0,T;H^{2})\quad\forall\,T>0
$$
emanating from this initial data. Furthermore, the solution verifies
$$
\|p(t)\|_{A^1}+\frac{\beta+1}{5\varepsilon}\int_0^t\|p(s)\|_{A^2}ds\leq C(p_0).
$$
\end{theorem}
\begin{proof}
As before, the well-posedness will follow from appropriate energy estimates and a regularization approach. As before, we start noticing that the zero-mean property is propagated in time. In order to obtain the global existence of solution, we start estimating $\|p\|_{A^0}$. We have that
$$
\pat |\hat{p}(t,k)|=\frac{\Re(\overline{\hat{p}}(t,k)\pat \hat{p}(t,k))}{|\hat{p}(t,k)|},
$$
so, using the inequality
$$
\|FG\|_{A^0}\leq \|F\|_{A^0}\|G\|_{A^0},
$$
we have that
$$
\frac{5}{4}\frac{d}{dt}\|p\|_{A^0}+\frac{\beta+1}{2\varepsilon}\|p\|_{A^1}\leq \|p\|_{A^0}\|p\|_{A^1}+\frac{1}{4}\|p^2\|_{A^1}.
$$
Using the triangle inequality to find
$$
\|p^2\|_{A^1}\leq \sum_{k}|k|\sum_{n}|\hat{p}(k-n)||\hat{p}(n)|\leq \sum_{k}\sum_{n}(|k-n|+|n|)|\hat{p}(k-n)||\hat{p}(n)|\leq 2\|p\|_{A^0}\|p\|_{A^1},
$$
we conclude
$$
\frac{5}{4}\frac{d}{dt}\|p\|_{A^0}+\frac{\beta+1}{2\varepsilon}\|p\|_{A^1}\leq \frac{3}{2}\|p\|_{A^0}\|p\|_{A^1}.
$$
Then, if the initial data is small enough, we conclude the estimate
$$
\|p(t)\|_{A^0}+\frac{\beta+1}{5\varepsilon}\int_0^t\|p(s)\|_{A^1}ds\leq C(p_0).
$$
Repeating the computation for $\partial_x p$, we find that
$$
\frac{5}{4}\frac{d}{dt}\|p\|_{A^1}+\frac{\beta+1}{2\varepsilon}\|p\|_{A^2}\leq \|p\|_{A^0}\|p\|_{A^2}+\|p\|_{A^1}^2+\frac{1}{4}\|p^2\|_{A^2}.
$$
We compute
$$
\|p^2\|_{A^2}=\sum_{k}|k|^2\sum_{n}|\hat{p}(k-n)||\hat{p}(n)|\leq \sum_{k}\sum_{n}C(|k-n|^2+|n|^2)|\hat{p}(k-n)||\hat{p}(n)|\leq C\|p\|_{A^0}\|p\|_{A^2}.
$$
We now observe that (see \cite{gancedo2020surface})
$$
\|p\|_{A^1}^2\leq C\|p\|_{A^0}\|p\|_{A^2}.
$$
Then, 
$$
\frac{d}{dt}\|p\|_{A^1}+\frac{2\beta+2}{5\varepsilon}\|p\|_{A^2}\leq C\|p\|_{A^0}\|p\|_{A^2},
$$
and, if the initial data is small enough,
$$
\|p(t)\|_{A^1}+\frac{\beta+1}{5\varepsilon}\int_0^t\|p(s)\|_{A^2}ds\leq C(p_0).
$$
Now we multiply \eqref{eq:alpha1} by $\pax^4p$ and integrate by parts to find
\begin{align*}
\frac{5}{8}\frac{d}{dt}\|p\|_{H^2}^2+\frac{\beta+1}{2\varepsilon}\|p\|_{H^{2.5}}^2\leq -\int_{\mathbb{T}}\pax^2\left(\frac{p^2}{2}\right)\pax^3 pdx+\int_{\mathbb{T}}\pax^2\left(\frac{p^2}{4}\right)(-\Delta)^{1/2}\pax^2 pdx.
\end{align*}
Further integrations by parts together with H\"older and Sobolev inequalities show that
$$
-\int_{\mathbb{T}}\pax^2\left(\frac{p^2}{2}\right)\pax^3 pdx\leq C\|\pax p\|_{L^\infty}\|\pax^2 p\|_{L^2}^2.
$$
The remainder nonlinear term can be estimated using a duality $H^{1/2}-H^{-1/2}$ argument as follows
$$
\frac{1}{2}\int_{\mathbb{T}}(p\pax^2 p+(\pax p)^2)(-\Delta)^{1/2}\pax^2 pdx\leq C\|p\pax^2 p\|_{H^{1/2}}\|(-\Delta)^{1/2}\pax^2 p\|_{H^{-1/2}}+C\|\pax^2 p\|_{L^2}\|(\pax p)^2\|_{H^1}
$$
From the previous inequality, we obtain that
\begin{align*}
\frac{5}{8}\frac{d}{dt}\|p\|_{H^2}^2+\frac{\beta+1}{2\varepsilon}\|p\|_{H^{2.5}}^2\leq C\|\pax p\|_{L^\infty}\|p\|_{H^{2.5}}^2\leq C\|p\|_{A^1}\|p\|_{H^{2.5}}^2.
\end{align*}
If the initial data is small enough then we conclude 
$$
\|p(t)\|_{H^2}^2+\frac{\beta+1}{\varepsilon}\int_0^t\|p(s)\|_{H^{2.5}}^2ds\leq C(p_0).
$$
This concludes with the global existence part. The uniqueness follows using a standard contradiction argument using the regularity of the solutions.
\end{proof}
A numerical study of the equation with values $N=2^{10}$ spatial nodes, $\varepsilon=1$, $\beta=2$ and initial data
$$
p(x,0)=-4\sin(10x)
$$
can be seen in figure \ref{figalpha1}. There the solution appears to exists globally and decay towards the flat equilibrium state. We think that is the case for initial data for which the linear part is dominant, however, we think that an ill-posedness result for large data should also be true. This is left for a future work.
\begin{figure}[h]
\begin{center}
\includegraphics[scale=0.3]{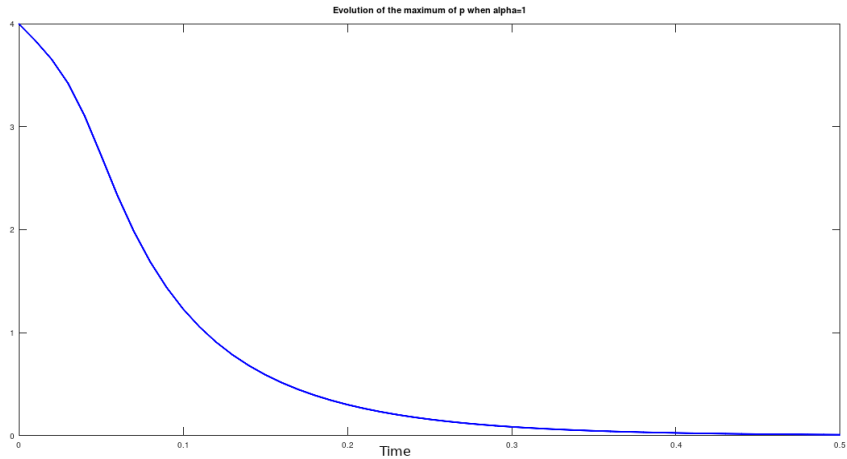}
\end{center}
\vspace{-.1 in}
\caption{$\|p(t)\|_{L^\infty}$ as a function of time.}
\label{figalpha1}
\end{figure}

\section{The case $\alpha=2$}
In this section we consider the case $\alpha=2$. Then equation \eqref{eq:Burgers3} reads
\begin{equation}\label{eq:alpha2}
\pat p-\frac{1}{4}\Delta\pat p=\frac{1+\beta}{2\varepsilon}\Delta  p-\frac{1}{4\varepsilon}\pax^3  p+p\pax p-\Delta\left(\frac{p^2}{4}\right)-\frac{\beta}{\varepsilon}\partial_x p.
\end{equation}

\begin{theorem}[Strong well-posedness for $\alpha=2$]
Let $p_0\in H^2$ be a zero-mean initial data, $\beta>-1$ and $\varepsilon>0$ be fixed constants. Then there exists a unique local solution to \eqref{eq:alpha2}
$$
p\in C([0,T_{max}),H^{2})\cap L^2(0,T_{max};H^{2})
$$
for a small enough $0<T_{max}\ll1$. Furthermore, there exists $0<c_0$ such that if
$$
\|p_0\|_{L^2}^2+\frac{1}{4}\|\pax p_0\|_{L^2}^2\leq c_0,
$$
then we have that there exists a unique global solution to \eqref{eq:alpha2}
$$
p\in C([0,T),H^2)\cap L^2(0,T;H^{2})\quad\forall\,T>0
$$
emanating from this initial data. Furthermore, the solution verifies
$$
\|\pax p(t)\|_{L^2}^2+\frac{1}{4}\|\Delta p(t)\|_{L^2}^2+\frac{\beta+1}{2\varepsilon}\int_0^t\|\Delta p(s)\|_{L^2}^2ds\leq C(p_0).
$$
\end{theorem}
\begin{proof}
We observe that the zero-mean property is propagated in time. We focus on obtaining the appropriate energy estimates. Testing \eqref{eq:alpha2} against $-\Delta p$ and integrating by parts, we find that
$$
\frac{d}{dt}\left(\|\pax p\|_{L^2}^2+\frac{1}{4}\|\Delta p\|_{L^2}^2\right)= -\frac{\beta+1}{\varepsilon}\|\Delta p\|_{L^2}^2-\int_\mathbb{T}p\pax p \Delta pdx+\int_\mathbb{T}\Delta\left(\frac{p^2}{4}\right)\Delta pdx.
$$
If we define now
$$
E(t)=\|\pax p(t)\|_{L^2}^2+\frac{1}{4}\|\Delta p(t)\|_{L^2}^2,
$$
using \eqref{eq:I2} and another integration by parts, we conclude the inequality
$$
\frac{d}{dt}E(t)\leq -\frac{\beta+1}{\varepsilon}\|\Delta p\|_{L^2}^2 + C(\|p\|_{L^\infty}+\|\pax p\|_{L^\infty})E(t)\leq E(t)^{3/2},
$$
from where the local existence follows using a classical regularization procedure (see, for instance, \cite{AGM,burczak2016generalized,GO}). The uniqueness follows using a standard contradiction argument using the regularity of the solutions. To obtain the global existence now we test equation \eqref{eq:alpha2} with $p$ and integrate by parts. We find that
$$
\frac{d}{dt}\left(\|p\|_{L^2}^2+\frac{1}{4}\|\pax p\|_{L^2}^2\right)= -\frac{\beta+1}{\varepsilon}\|\pax p\|_{L^2}^2+\int_\mathbb{T}\Delta\left(\frac{p^2}{4}\right) pdx.
$$
We can also compute
$$
\int_\mathbb{T}\Delta\left(\frac{p^2}{4}\right) pdx=-\frac{1}{2}\int_\mathbb{T} p(\pax p)^2dx\leq C\|p\|_{L^\infty}\|\pax p\|_{L^2}^2
$$
Furthermore, if we define
$$
F(t)=\|p(t)\|_{L^2}^2+\frac{1}{4}\|\pax p(t)\|_{L^2}^2,
$$
Sobolev embedding and Young's inequality lead us to 
$$
\|p\|_{L^\infty}\leq C\sqrt{F(t)},
$$
so we also find that
$$
\frac{d}{dt}F(t)\leq \left(C\sqrt{F(t)}-\frac{1+\beta}{\varepsilon}\right)\|\pax p\|_{L^2}^2,
$$
and we conclude the global uniform bound in $H^1$
$$
F(t)\leq F(0),
$$
for small enough initial data in $H^1$. Once the global bound in $H^1$ is achieved, we turn our attention to the previous estimates in $H^2$. A finer study together with Poincar\'e inequality shows that
$$
\frac{d}{dt}E(t)\leq -\frac{\beta+1}{\varepsilon}\|\Delta p\|_{L^2}^2+C\|p\|_{L^\infty}\|\Delta p\|_{L^2}^2.
$$
As a consequence
$$
\frac{d}{dt}E(t)\leq \left(C\sqrt{F(0)}-\frac{\beta+1}{\varepsilon}\right)\|\Delta p\|_{L^2}^2.
$$
From where we can conclude the global existence for small data with a standard continuation argument.
\end{proof}
Equation \eqref{eq:alpha2} can be equivalently written as
$$
\pat p=\frac{1+\beta}{2\varepsilon}\mathcal{J}\Delta  p-\frac{1}{4\varepsilon}\mathcal{J}\pax^3  p+\mathcal{J}(p\pax p)-\mathcal{J}\Delta\left(\frac{p^2}{4}\right)-\frac{\beta}{\varepsilon}\mathcal{J}\partial_x p.
$$
with
$$
\widehat{\mathcal{J}}(k)=\frac{1}{1+\frac{k^2}{4}}.
$$
Using this formulation, we can run simulations using the previously mentioned Fourier collocation to discretize in time and Runge-Kutta 4-5 to advance in time. Then, if we fix $N=2^{12}$ spatial nodes, $\varepsilon=1$, $\beta=2$ and initial data
$$
p(x,0)=-6\sin(4x^2),
$$
we obtain the plots \ref{figalpha2}. We see that the solution seem to exists globally and to decay towards the flat equilibrium. This is also the case for a number of different initial data that we also considered. Based on this we are tempted to say that the solution is probably globally defined regardless of the size of the initial data, however, the proof of this claim is left for a future work.
\begin{figure}[h]
\begin{center}
\includegraphics[scale=0.4]{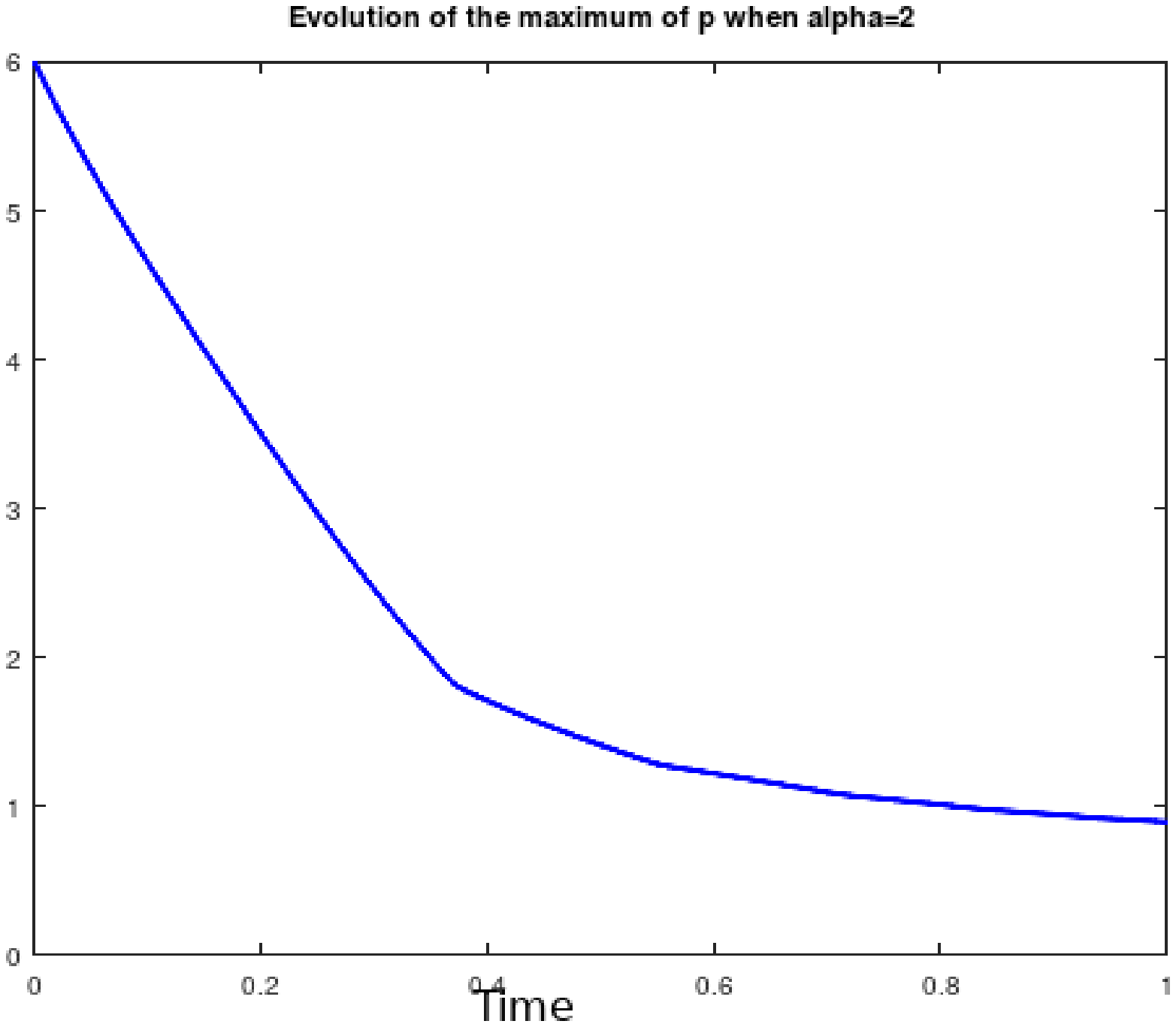} \includegraphics[scale=0.4]{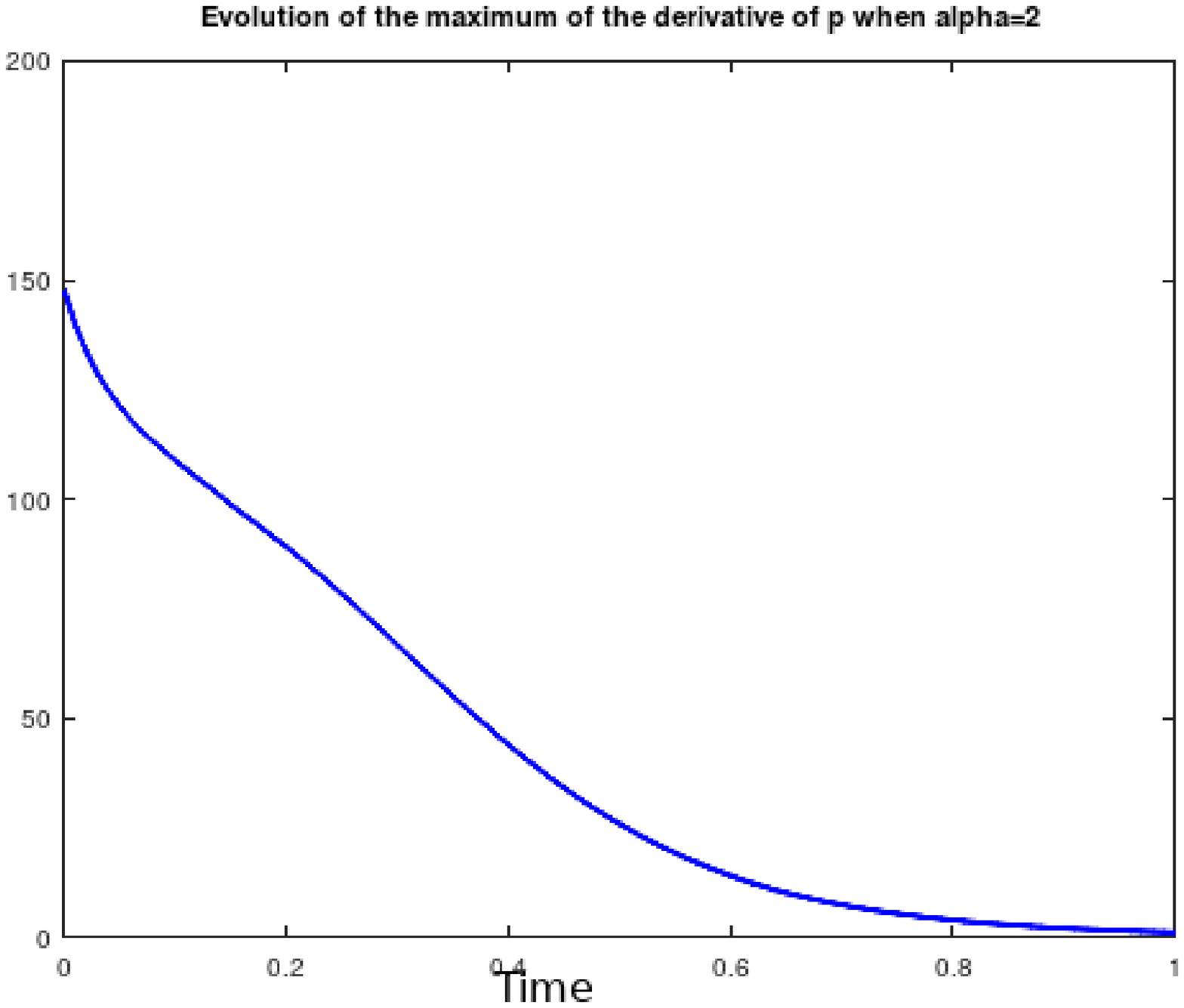}
\end{center}
\vspace{-.1 in}
\caption{a) $\|p(t)\|_{L^\infty}$ as a function of time. b) $\|\pax p(t)\|_{L^\infty}$ as a function of time.}
\label{figalpha2}
\end{figure}
\section{The case with general $\alpha$} 
In this section we prove the well-posedness of \eqref{eq:Burgers3} for general value of $\alpha$:
\begin{theorem}
Let $\alpha\geq0$, $\beta>-1$ and $\varepsilon>0$ be fixed constants. Define
$$
r=\max\{2,1+\alpha\}.
$$
Let $p_0\in H^r$ be a zero-mean initial data. There exists $0<c_0$ such that if
$$
\|p_0\|_{H^2}^2+\frac{1}{4}\|(-\Delta)^{(\alpha-1)/2}p_0\|_{H^2}^2\leq c_0,
$$
then we have that there exists a unique global solution to \eqref{eq:Burgers3}
$$
p\in C([0,T),H^{r})\cap L^2(0;T;H^{2+\frac{\alpha}{4}})\quad\forall\,T>0
$$
emanating from this initial data. Furthermore, the solution verifies
$$
\|p(t)\|_{H^2}^2+\frac{1}{4}\|(-\Delta)^{(\alpha-1)/2}p(t)\|_{H^2}^2+\frac{\beta+1}{2\varepsilon}\int_0^t\|(-\Delta)^{\alpha/4}\partial_x^2p(s)\|_{L^2}^2ds\leq C(p_0).
$$
\end{theorem}
\begin{proof}
As before, we focus on obtaining appropriate energy estimates. Similarly, the solution maintains the zero-mean property. Multiplying \eqref{eq:Burgers3} by $p$ and integrating by parts to find
\begin{align*}
\frac{1}{2}\frac{d}{dt}\left(\|p\|_{L^2}^2+\frac{1}{4}\|(-\Delta)^{(\alpha-1)/2}p\|_{L^2}^2\right)&=-\frac{\beta+1}{2\varepsilon}\|(-\Delta)^{\alpha/4}p\|_{L^2}^2+\int (-\Delta)^{\alpha/2}\left(\frac{p^2}{4}\right)pdx\\
&\leq -\frac{\beta+1}{2\varepsilon}\|(-\Delta)^{\alpha/4}p\|_{L^2}^2+C\|p\|_{L^\infty}\|(-\Delta)^{\alpha/4}p\|_{L^2}^2,
\end{align*}
where we have used the fractional Leibniz rule
$$
\|(-\Delta)^{s/2}(FG)\|_{L^q}\leq C\left(\|(-\Delta)^{s/2} F\|_{L^{q_1}}\|G\|_{L^{q_2}}\right.\\
\left.+\|(-\Delta)^{s/2} G\|_{L^{q_3}}\|F\|_{L^{q_4}}\right),
$$
with $s>\max\{0,1/q-1\}$
$$
\frac{1}{q}=\frac{1}{q_1}+\frac{1}{q_2}=\frac{1}{q_3}+\frac{1}{q_4}\qquad \mbox{where $1/2<q<\infty,1<p_i\leq\infty$}.
$$
Similarly, if we now multiply \eqref{eq:Burgers3} by $\partial_x^4 p$, we obtain that 
$$
\frac{1}{2}\frac{d}{dt}\left(\|\pax^2 p\|_{L^2}^2+\frac{1}{4}\|(-\Delta)^{(\alpha-1)/2}\pax^2 p\|_{L^2}^2\right)=-\frac{\beta+1}{2\varepsilon}\|(-\Delta)^{\alpha/4}\partial_x^2p\|_{L^2}^2+I_1+I_2
$$
with
$$
I_1=\int_\mathbb{T}p\partial_xp \partial_x^4pdx\leq \frac{5}{2}\|\partial_xp\|_{L^\infty}\|\partial_x^2 p\|_{L^2}^2
$$
$$
I_2=\int_\mathbb{T}(-\Delta)^{\alpha/2}\left(\frac{p^2}{4}\right)\partial_x^4pdx,
$$
where we have used \eqref{eq:I1}. Similarly, we compute that
\begin{align*}
I_2&=-\int_\mathbb{T}(-\Delta)^{\alpha/4+1}\left(\frac{p^2}{4}\right)(-\Delta)^{\alpha/4}\partial_x^2pdx\\
&\leq C\|p\|_{L^\infty}\|(-\Delta)^{\alpha/4}\partial_x^2p\|_{L^2}^2
\end{align*}
As a consequence, we conclude that
\begin{align*}
\frac{d}{dt}\left(\|p\|_{H^2}^2+\frac{1}{4}\|(-\Delta)^{(\alpha-1)/2}p\|_{H^2}^2\right)&\leq -\frac{\beta+1}{\varepsilon}\|(-\Delta)^{\alpha/4}p\|_{L^2}^2+C\|\pax p\|_{L^\infty}\|(-\Delta)^{\alpha/4}p\|_{L^2}^2\\
&\quad-\frac{\beta+1}{\varepsilon}\|(-\Delta)^{\alpha/4}\partial_x^2p\|_{L^2}^2+C\|\pax p\|_{L^\infty}\|(-\Delta)^{\alpha/4}\partial_x^2p\|_{L^2}^2\nonumber.
\end{align*}
Using the Sobolev embedding, we find that
$$
\|\pax p\|_{L^\infty}\leq C\|p\|_{H^{3/2+\delta}}\,\forall\,\delta>0.
$$
Taking $\delta=1/2$ we conclude that
\begin{align*}
\frac{d}{dt}\left(\|p\|_{H^2}^2+\frac{1}{4}\|(-\Delta)^{(\alpha-1)/2}p\|_{H^2}^2\right)&\leq -\frac{\beta+1}{\varepsilon}\|(-\Delta)^{\alpha/4}p\|_{L^2}^2+C\|p\|_{H^2}\|(-\Delta)^{\alpha/4}p\|_{L^2}^2\\
&\quad-\frac{\beta+1}{\varepsilon}\|(-\Delta)^{\alpha/4}\partial_x^2p\|_{L^2}^2+C\|p\|_{H^2}\|(-\Delta)^{\alpha/4}\partial_x^2p\|_{L^2}^2\nonumber\\
&\leq -\frac{\beta+1}{\varepsilon}\|(-\Delta)^{\alpha/4}p\|_{L^2}^2\nonumber\\
&\quad+C\left(\|p\|_{H^2}^2+\frac{1}{4}\|(-\Delta)^{(\alpha-1)/2}p\|_{H^2}^2\right)^{1/2}\|(-\Delta)^{\alpha/4}p\|_{L^2}^2\nonumber\\
&\quad-\frac{\beta+1}{\varepsilon}\|(-\Delta)^{\alpha/4}\partial_x^2p\|_{L^2}^2\nonumber\\
&\quad+C\left(\|p\|_{H^2}^2+\frac{1}{4}\|(-\Delta)^{(\alpha-1)/2}p\|_{H^2}^2\right)^{1/2}\|(-\Delta)^{\alpha/4}\partial_x^2p\|_{L^2}^2\nonumber.
\end{align*}
And, if the initial data is small enough, we find
$$
\frac{d}{dt}\left(\|p\|_{H^2}^2+\frac{1}{4}\|(-\Delta)^{(\alpha-1)/2}p\|_{H^2}^2\right)+\frac{\beta+1}{2\varepsilon}\|(-\Delta)^{\alpha/4}\partial_x^2p\|_{L^2}^2\leq0,
$$
from where we conclude the global existence. 

The uniqueness follows using a standard contradiction argument using the regularity of the solutions.
\end{proof}

\section*{Acknowledgments}
R.G-B was supported by the project "Mathematical Analysis of Fluids and Applications" Grant PID2019-109348GA-I00 funded by MCIN/AEI/ 10.13039/501100011033 and acronym "MAFyA". This publication is part of the project PID2019-109348GA-I00 funded by MCIN/ AEI /10.13039/501100011033. R.G-B is also supported by a 2021 Leonardo Grant for Researchers and Cultural Creators, BBVA Foundation. The BBVA Foundation accepts no responsibility for the opinions, statements, and contents included in the project and/or the results thereof, which are entirely the responsibility of the authors. The author thanks Martina Magliocca for her helpful comments that greatly improve the final version of the manuscript.

\bibliographystyle{plain}

\end{document}